\documentclass{article}
\usepackage{amsfonts}
\usepackage{latexsym}
\usepackage{mathrsfs}
\usepackage{amssymb}
\usepackage{amsmath}
\usepackage{amsthm}
\usepackage{indentfirst}
\usepackage{color}
\usepackage{float}
\usepackage{graphicx}

\allowdisplaybreaks
\newtheorem{theorem}{\color{black}\indent Theorem}[section]
\newtheorem{lemma}{\color{black}\indent Lemma}[section]

\newtheorem{definition}{\color{black}\indent Definition}[section]
\newtheorem{remark}{\color{black}\indent Remark}[section]

\begin{document}
\title{\LARGE\bf Existence of solutions to a perturbed critical biharmonic equation with Hardy potential}
\author{Qi Li$^1$ \qquad Yuzhu Han$^1$\qquad Jian Wang$^{2,\dag}$}
 \date{}
 \maketitle

\footnotetext{\hspace{-1.9mm}$^\dag$Corresponding author.\\
Email addresses: pdejwang@ouc.edu.cn(J. Wang).

\thanks{
$^*$Supported by the National Key Research and Development Program of China (grant no.2020YFA0714101).}}
\begin{center}
{\it\small $^1$ School of Mathematics, Jilin University, Changchun 130012, P.R. China}\\
{\it\small $^2$ School of Mathematical Sciences, Ocean University of China, Qingdao 266100, P.R. China}
\end{center}
\date{}
\maketitle

{\bf Abstract}\ In this paper, the following biharmonic elliptic problem
\begin{eqnarray*}
\begin{cases}
\Delta^2u-\lambda\frac{|u|^{q-2}u}{|x|^s}=|u|^{2^{**}-2}u+ f(x,u), &x\in\Omega,\\
u=\dfrac{\partial u}{\partial n}=0, &x\in\partial\Omega
\end{cases}
\end{eqnarray*}
is considered. 
The main feature of the equation is that it involves a Hardy term and a
nonlinearity with critical Sobolev exponent. By combining a careful analysis of the fibering maps
of the energy functional associated with the problem with the Mountain Pass Lemma,
it is shown, for some positive parameter $\lambda$ depending on $s$ and $q$,
that the problem admits at least one mountain pass type solution under appropriate growth conditions on the nonlinearity $f(x,u)$.




\section{Introduction}
\setcounter{equation}{0}

In this paper, we are concerned with the following biharmonic elliptic problem
\begin{eqnarray}\label{P1.1}
\ \ \ \ \ \ \ \ \ \ \ \ \begin{cases}
\Delta^2u-\lambda\dfrac{|u|^{q-2}u}{|x|^s}=|u|^{2^{**}-2}u+ f(x,u), &x\in\Omega,\\
u=\dfrac{\partial u}{\partial n}=0, &x\in\partial\Omega,\\
\end{cases}
\end{eqnarray}
where $\Omega\subset \mathbb{R}^{N}(N\geq5)$ is a bounded domain containing $0$ with smooth boundary
$\partial\Omega$, $2\leq q\leq2^{**}(s):=\frac{2(N-s)}{N-4}< 2^{**}:=\frac{2N}{N-4}$ for $0< s\leq4$,
$2\leq q<2^{**}$ for $s=0$, $\lambda$ is a positive parameter, $\Delta^2$ is the biharmonic operator,
and $f(x,t)$: $\overline{\Omega}\times \mathbb{R}\rightarrow \mathbb{R}$ is a continuous  function
and is odd with respect  to $t$ which satisfies the following two assumptions $(f_1)$ and $(f_2)$:

($f_1$) \ $\lim\limits_{t\rightarrow0^+}\dfrac{f(x,t)}{t}=0$ and
$\lim\limits_{t\rightarrow+\infty}\dfrac{f(x,t)}{t^{2^{**}-1}}=0$ uniformly in $x\in \Omega$;

($f_2$) there exists a constant $\rho\in(q,2^{**})$ such that $0<\rho F(x,t)\leq tf(x,t)$ for all
$x\in\Omega$ and $t\in \mathbb{R}\backslash\{0\}$, where $F(x,t)=\int_{0}^{t}f(x,\tau)\mathrm{d}\tau$.
This condition is usually referred to as the Ambrosetti-Rabinowitz superlinear condition.


According to $(f_1)$, we observe that, for any $\varepsilon>0$ ,
there exist $C_i(\varepsilon)>0, (i=1,2,3)$ such that
\begin{equation}\label{Mm}
|f(x,t)|\leq\varepsilon t^{2^{**}-1}+C_1(\varepsilon) t,\ x\in\Omega, \ t\in \mathbb{R},
\end{equation}
\begin{equation*}
|f(x,t)|\leq \varepsilon t+C_2(\varepsilon)t^{2^{**}-1},\ x\in\Omega, \ t\in \mathbb{R},
\end{equation*}
and
\begin{equation}\label{Mmm}
 |F(x,t)|\leq \dfrac{1}{2}\varepsilon t^{2}+C_3(\varepsilon)t^{2^{**}},\ x\in\Omega, \ t\in \mathbb{R}.
\end{equation}
As a consequence of ($f_2$), one sees that there exists a constant $C>0$ such that
\begin{equation}\label{Mmmm}
F(x, t)\geq C|t|^\rho,
\end{equation}
 for all $x\in\Omega$ and $t\in \mathbb{R}$.


Due to their wide application in describing a variety of phenomena in physics and other applied sciences,
both local and nonlocal elliptic problems with critical exponents have been extensively studied in recent years,
and remarkable progress has been made on the existence, non-existence and multiplicity of weak solutions to these problems.
For example, Q. Xie et al. \cite{Q. Xie} dealt with the following Kirchhoff type problem involving  critical exponent
\begin{equation}\label{1.5}
\begin{cases}
-(a+b\displaystyle{\int_{\Omega}|\nabla u|^2\mathrm{d}x})\Delta u=f(x, u)+u^5,\quad &x\in\Omega,\\
u(x)=0,\quad& x\in\partial\Omega,
\end{cases}
\end{equation}
where $\Omega\subset \mathbb{R}^3$ is a smooth bounded domain, $a\geq0$, $b>0$ and $f(x,u)$:
$\overline{\Omega}\times \mathbb{R}\rightarrow \mathbb{R}$ is a continuous function.
The existence and multiplicity of solutions to problem \eqref{1.5} were obtained by using variational method.
When the dimension of the space is $4$, D. Naimen \cite{2} considered the following Kirchhoff type problem involving  critical exponent
\begin{equation}\label{55}
\begin{cases}
-(a+b\displaystyle{\int_{\Omega}|\nabla u|^2\mathrm{d}x})\Delta u=\lambda u^q+\mu u^3,\quad  u>0,\quad &x\in\Omega,\\
u(x)=0,\quad& x\in\partial\Omega,
\end{cases}
\end{equation}
where $a,\lambda,\mu>0$, $b\geq0$ and $1\leq q<3$. The existence and non-existence of weak solutions to
problem \eqref{55} were obtained, also by using variational method.

There are also some works dealing with biharmonic equations with critical exponents.
For the case $q=2$, problem \eqref{P1.1} with $\lambda=1$, $0\leq s\leq2$ and
$f(x,u)\equiv0$ was considered by Kang et al. in \cite{Kang}. By using the Sobolev-Hardy
inequality and variational method, they showed that problem \eqref{P1.1} has at least
one nontrivial solution when $N\geq8-s$. Later, the results of \cite{Kang} was generalized by Li et al. in \cite{li},
where problem \eqref{P1.1} with $\lambda>0$, $0\leq s\leq 4$ and $f(x,u)=f(x)$ was investigated.
They proved that there exist at least two nontrivial solutions when the norm of $f$ is appropriately small
and the parameters $\lambda$ and $s$ satisfy some appropriate conditions.
When $q\geq2$, problem \eqref{P1.1} with $f(x,u)\equiv0$ has also been studied.
For example, Yao et al. \cite{4Hardylinj} considered problem \eqref{P1.1} with $0\leq s\leq4$,
and obtained the existence and non-existence of nontrivial solutions to problem \eqref{P1.1},
with the help of Sobolev-Hardy inequality and the Mountain Pass Lemma.
It is worth mentioning that there are many other interesting works on nonlinear elliptic problems with critical exponents,
among the huge amount of which, we only refer the interested reader to
\cite{C. O. Alves,A. Ambrosetti,D. Cao,L. D'Ambrosio,A. Fiscella,X. He,2,1,J. Wang,X. Wang,X. Mingqi,X. Zhong} and the references therein.

Motivated mainly by \cite{Kang,li,4Hardylinj}, it is natural to consider the existence
of weak solutions to problem \eqref{P1.1} with $q\geq2$ and general nonlinearity $f$
that depends on both $x$ and the unknown function $u$.
As far as we know, there have been few works in this direction, mainly because of the following two difficulties.
The first one is the lack of compactness of the Sobolev embedding $H_0^2(\Omega)\hookrightarrow L^{2^{**}}(\Omega)$,
which prevents us from directly establishing the Palais-Smale condition for the associated energy functional.
The second one is caused
by the Hardy term, since we can not establish the compactness of the mapping $u\rightarrow\dfrac{u}{|x|^{2}}$ from $H_0^2(\Omega)$ into $L^2(\Omega)$ when $s=4$.

In this paper, the above two difficulties are overcome by combining a careful analysis of the fibering maps
of the energy functional associated with the problem with the Mountain Pass Lemma \cite{M. Willem} and
Br\'{e}zis-Lieb's lemma \cite{H. Br}, and a mountain pass type solution to problem \eqref{P1.1} is obtained.
Let us explain our strategy in a more detailed way. When $0\leq s<4$, with the help of  Br\'{e}zis-Lieb's lemma,
we  prove that the energy functional associated with problem \eqref{P1.1}
satisfies the $(PS)_c$ condition for $c<c(S)$ (Lemma \ref{lem3}).
Then, after some careful estimates on the norms of the truncated Talenti functions,
we show that the energy functional has a mountain pass geometry around $0$ and that
the corresponding mountain pass level is small than $c(S)$.
Finally, on the basis of the above two steps, a mountain pass type solution to problem \eqref{P1.1} follows from a standard variational approach.
When $s=4$, by introducing an equivalent norm in $H_0^2(\Omega)$ (Lemma \ref{w}) and applying a variant of Br\'{e}zis-Lieb's lemma,
we can also show that the energy functional satisfies the $(PS)_c$ condition for $c<c(S_\lambda)$ (Lemma \ref{lem2.1}).
Then the mountain pass level is proved to be strictly smaller than $c(S_\lambda)$ with the help of another group of Talenti functions
and the existence of a mountain pass type solution follows.

The remainder of this paper is organized as follows.
In Section 2 we give some notations, definitions and introduce some necessary lemmas.
The main results of this paper are also stated in this section.
In Section 3 we give the detailed proof of the main results.

\par
\section{Preliminaries}
\setcounter{equation}{0}

In this section, we first introduce some notations and definitions that will be used throughout the paper.
In what follows, we denote by $\|\cdot\|_{p}$ the $L^p(\Omega)$ norm for $1\leq p\leq \infty$,
and denote the norm of the weighted space $L^p(\Omega, |x|^{-s})$ by
\begin{equation*}
\|\cdot\|_{L^p(\Omega, |x|^{-s})}=\big(\int_{\Omega}|x|^{-s} |\cdot|^p\mathrm{d}x\big)^{\frac{1}{p}},\ 1\leq p< \infty.
\end{equation*}
  The Sobolev space $H_0^2(\Omega)$ will be equipped with the norm $\|u\|:=\|u\|_{H_0^2(\Omega)}=\|\Delta u\|_2$,
which is equivalent to the full one due to Poincar\'{e}'s inequality, and its dual space is denoted by $H^{-2}(\Omega)$.
We always use $\rightarrow$ and $\rightharpoonup$ to denote the strong and weak convergence in each Banach space, respectively,
and use $C$, $C_1$, $C_2$,$...$ to denote generic positive constants.
$B_r(x_0)$ is a ball of  radius $r$ centered at $x_0$.
For each $t>0$, $O(t)$ denotes
the quantity satisfying $|\frac{O(t)}{t}|\leq C$, $O_1(t)$ means that there exist two positive constants
$C_1$ and $C_2$ such that $C_1t\leq O_1(t)\leq C_2t$, and $o(t)$ means that $|\frac{o(t)}{t}|\rightarrow0$ as $t\rightarrow0$.


In this paper, we consider weak solutions to problem \eqref{P1.1} in the following sense.
\begin{definition}\label{de2.1}$\mathrm{\bf{(Weak \ solution)}}$
A function $u\in  H_{0}^{2}(\Omega)$ is called a weak solution to problem \eqref{P1.1},
if for all $\varphi\in H_{0}^{2}(\Omega)$, it holds that
\begin{equation*}\label{eq1}
\int_{\Omega}\Delta u\Delta \varphi\mathrm{d}x-\lambda\int_{\Omega}\frac{|u|^{q-2}u\varphi}{|x|^s}\mathrm{d}x
-\int_{\Omega}|u|^{2^{**}-2}u\varphi\mathrm{d}x-\int_{\Omega}f(x,u)\varphi\mathrm{d}x=0.
\end{equation*}
\end{definition}
The energy functional associated with problem \eqref{P1.1} is given by
\begin{equation*}\label{eq2}
I_\lambda(u)=\frac{1}{2}\| u\|^2-\frac{\lambda}{q}\int_{\Omega}\frac{|u|^q}{|x|^s}\mathrm{d}x-\frac{1}{2^{**}}\|u\|_{2^{**}}^{2^{**}}-\int_{\Omega}F(x,u)\mathrm{d}x,\ \forall \ u\in H_{0}^{2}(\Omega).
\end{equation*}
According to the hypotheses $(f_1)-(f_2)$ and $2\leq q<2^{**}$, it is directly verified that
$I_\lambda(u)$ is a $C^1$ functional in $H_{0}^{2}(\Omega)$ (see \cite{li}).
Since $f$ is odd with respect to $t$, which implies that $I_\lambda(u)=I_\lambda(|u|)$,
we may assume that $u\geq0$ in the sequel.

We then introduce a compactness condition known as local (PS) condition or the $(PS)_c$ condition,
which will assist us in finding weak solutions to problem \eqref{P1.1}.


\begin{definition}\label{compactness}($(PS)_c$ condition)
Assume that $X$ is a real Banach space, $I:X\rightarrow \mathbb{R}$ is a $C^1$ functional and $c\in \mathbb{R}$.
We say that $I$ satisfies the $(PS)_c$ condition if any sequence $\{u_n\}\subset X$ such that
\begin{eqnarray*}
I(u_n)\rightarrow c\  in \ \mathbb{R} \ and \ I'(u_{n})\rightarrow 0\ in  \ X^{-1}(\Omega)\ as \ n\rightarrow\infty
\end{eqnarray*}
has a convergent subsequence, where $X^{-1}$ is the dual space of $X$.
\end{definition}

The following three lemmas are crucial in our analysis. The first one is the famous Mountain Pass Lemma,
the second one is the Br\'{e}zis-Lieb's lemma and the third one is its variant.

\begin{lemma}\label{lemM}(Mountain Pass Lemma \cite{M. Willem})
Assume that $(X, \|\cdot\|_X)$ is a real Banach space, $I:X\rightarrow \mathbb{R}$ is a $C^1$ functional
and there exist $\beta>0$ and $r>0$ such that $I$ satisfies the following  mountain pass geometry:

(i) $I(u)\geq \beta>0$ if $\|u\|_X=r$;

(ii) there exists a $\overline{u}\in X$ such that $\|\overline{u}\|_X>r$ and $I(\overline{u})<0$.

Then there exist a sequence $\{u_n\}\subset X$ such that $I(u_n)\rightarrow c_0$ in $\mathbb{R}$ and $I'(u_{n})\rightarrow 0$ in $X^{-1}$ as $n\rightarrow\infty$,
where $X^{-1}$ is the dual space of $X$ and
\begin{eqnarray*}
c_0:=\inf_{\gamma\in\Gamma}\max_{t\in[0,1]}I(\gamma(t))\geq \beta,
 \ \Gamma=\left\{\gamma\in C([0,1],X): \gamma(0)=0, \gamma(1)=\overline{u}\right\},
\end{eqnarray*}
which is called the mountain level. Furthermore, $c_0$ is a critical vale of $I$ if $I$ satisfies the $(PS)_{c_0}$ condition.
\end{lemma}

\begin{lemma}\label{Lieb}(Br\'{e}zis-Lieb's lemma \cite{H. Br})
Let 
 $p\in(1,\infty)$. Suppose that $\{u_n\}$
is a bounded sequence in $L^p(\Omega)$ and $u_n\rightarrow u$ a.e. in  $\Omega$.  Then
\begin{eqnarray*}
\begin{split}
&\lim_{n\rightarrow\infty}(\|u_n\|_p^p-\|u_n-u\|_p^p)=\|u\|_p^p.\\
\end{split}
\end{eqnarray*}
\end{lemma}

\begin{lemma}\label{lemB}(\cite{G. Li})
Let $r>1$, $q\in[1,r]$ and $\delta\in[0,\dfrac{Nq}{r})$. If $\{u_n\}$
is a bounded sequence in $L^r(\mathbb{R}^N, |x|^{-\frac{\delta r}{q}})$ and $u_n\rightarrow u$ a.e. in $\mathbb{R}^N$.
Then
\begin{eqnarray*}
\begin{split}
&\lim_{n\rightarrow\infty}\int_{\mathbb{R}^N}\Big|\frac{|u_n|^q}{|x|^\delta}
-\frac{|u_n-u|^q}{|x|^\delta}-\frac{|u|^q}{|x|^\delta}\Big|^\frac{r}{q}\mathrm{d}x=0,\\
&\lim_{n\rightarrow\infty}\int_{\mathbb{R}^N}\Big|\frac{|u_n|^{q-1}u_n}{|x|^\delta}
-\frac{|u_n-u|^{q-1}(u_n-u)}{|x|^\delta}-\frac{|u|^{q-1}u}{|x|^\delta}\Big|^\frac{r}{q}\mathrm{d}x=0.
\end{split}
\end{eqnarray*}
\end{lemma}

The main results of the paper are summarized into the theorem below, which brings this part to a close.

\begin{theorem}\label{th} Assume that both $(f_1)$ and $(f_2)$ hold.

$(i)$ For $0\leq s< 4$ and $q=2$, if $N\geq8-s$ or $\max\{\frac{N}{N-4},\frac{8}{N-4}\}<\rho<2^{**}$,
then problem \eqref{P1.1} has at least one nonnegative solution for $0<\lambda<\lambda_{s,2}$,
where $\lambda_{s,2}$ is given in \eqref{Vv}.

$(ii)$ For $0\leq s< 4$ and   $2<q<2^{**}(s)$, if $q>\max\{\frac{N-s}{N-4},\frac{2(4-s)}{N-4}\}$
or $\max\{\frac{N}{N-4},\frac{8}{N-4}\}<\rho<2^{**}$, then problem \eqref{P1.1} has at least one
nonnegative solution for all $\lambda>0$.


$(iii)$ For $s=4$ and $q=2$, if $\max\{\frac{N}{\gamma_\lambda}, \frac{4(N-2-\gamma_\lambda)}{N-4}\}<\rho<2^{**}$,
then problem \eqref{P1.1} has at least one nonnegative solution for $\lambda\in(0,\lambda_{4,2})$,
where $\gamma_\lambda$ and $\lambda_{4,2}$ are given in Lemma \ref{r} and Remark \ref{remark1}, respectively.
\end{theorem}

\begin{remark}\label{remark2}  A typical example of a function that meets the conditions  $(f_1)-(f_2)$ is
\begin{equation*}
f(x,t)=\sum_{i=1}^{k}C_i(x)|t|^{q_i-2}t,\ for\ x \in\overline{\Omega}, \ t \in \mathbb{R},
\end{equation*}
where $k\in \mathbb{N}$, $q< q_i\leq2^{**}(s)$ for $0<s\leq 4$, $q< q_i<2^{**}$ for $s=0$, and $C_i\in C(\overline{\Omega})$ is positive.
\end{remark}

\par
\section{Proofs of the main results.}
\setcounter{equation}{0}


This section starts off with a lemma that provides the Sobolev-Hardy inequality, 
which is crucial for getting the $(PS)_c$ condition of functional $I_\lambda$. The proof of this result may be found, for instance, in  \cite{4Hardylinj}.

\begin{lemma}\label{lem1}
Assume that $2\leq q\leq2^{**}(s)=\frac{2(N-s)}{N-4}$ $(0\leq s\leq4)$. Then

$(i)$ there exists a constant $C>0$ such that
  $C\big(\displaystyle{\int_\Omega\dfrac{|u|^q}{|x|^s}\mathrm{d}x}\big)^{\frac{1}{q}}\leq \| u\|$ for any $u\in H_0^2(\Omega)$;

$(ii)$ if $2\leq q<2^{**}(s)$,
the mapping $u\rightarrow\dfrac{u}{|x|^{\frac{s}{q}}}$ from $H_0^2(\Omega)$ into $L^q(\Omega)$ is compact.
\end{lemma}

\begin{remark}\label{remark1}
Denote the best Sobolev-Hardy constant
\allowdisplaybreaks  \begin{align}
\label{Vv}\lambda_{s,q}=\inf_{u\in H_0^2(\Omega)\backslash\{0\}}\frac{\| u\|^2}{(\int_\Omega\frac{|u|^q}{|x|^s}\mathrm{d}x)^{\frac{2}{q}}}.
\end{align}
In particular,  $\lambda_{4,2}=\dfrac{1}{16}N^2(N-4)^2$. We always write $\lambda_{0,2^{**}}$ as $S$ for simplicity,
which satisfies $\|u\|_{2^{**}}^{2^{**}}\leq S^{-\frac{{2^{**}}}{2}}\|u\|^{2^{**}}$.
\end{remark}

In order to prove Theorem \ref{th}, the following result, which follows immediately from \eqref{Vv} with $q=2$, is needed.

\begin{lemma}\label{w}
For $\lambda\in (0,\lambda_{s, 2})$ $(0\leq s\leq4)$, there exists a $\mu_s>0$ such that
\begin{equation}\label{Hhh}
\displaystyle{\int_{\Omega}(|\Delta u|^2-\lambda\dfrac{|u|^2}{|x|^s})\mathrm{d}x}\geq\mu_s\| u\|^2,
\end{equation}
for any $u\in H_0^2(\Omega)$.
\end{lemma}

\begin{remark}\label{SS}
From Lemma \ref{w} it follows that the following 
best Sobolev constant
is well defined  for $\lambda\in(0,\lambda_{4,2})$
\begin{equation}\label{Hh}
S_\lambda=\inf_{u\in H_0^2(\Omega)\backslash\{0\}}\frac{\int_{\Omega}(|\Delta u|^2-\lambda\frac{|u|^2}{|x|^4})\mathrm{d}x}{\big(\int_{\Omega}| u|^{2^{**}}\mathrm{d}x\big)^{\frac{2}{2^{**}}}}>0.
\end{equation}
\end{remark}

\subsection{The case $0\leq s<4$.}

In general, the functional $I_\lambda(u)$  does not satisfy the $(PS)_c$ condition for all $c\in \mathbb{R}$,
due to the appearance of the critical term. However, with the help of Br\'{e}zis-Lieb's lemma,
we can find a constant $c(S)$ such that the $(PS)_c$ condition holds for all $c<c(S)$.
This will be essential in revealing the main results.

\begin{lemma}\label{lem3}
Assume that $f(x,t)$ satisfies $(f_1)$ and $(f_2)$, $0\leq s<4$.
Let $\{u_{n}\}\subset H_0^2(\Omega)$ be a sequence such that $I_\lambda(u_{n})\rightarrow c<c(S)$ and $I_\lambda'(u_{n})\rightarrow 0$ in $H^{-2}(\Omega)$ as $n\rightarrow\infty$, where $c(S):=\dfrac{2}{N}S^{\frac{N}{4}}$. Then, $I_\lambda(u)$  satisfies the $(PS)_c$ condition if $q=2$ and $0<\lambda<\lambda_{s,2}$, or $2< q<2^{**}(s)$ and $\lambda>0$.
\end{lemma}

\begin{proof}
We begin the proof with showing the boundedness of $\{u_n\}$ in $H_0^2(\Omega)$.
Depending on whether or not $q$ is equal to $2$, the proof will be divided into two cases.
When $q>2$,  by virtue of  $(f_2)$, we obtain, for any $\lambda>0$, that
\allowdisplaybreaks  \begin{align*}
c+1+o(1)\|u_n\|
&\geq I_\lambda(u_n)-\dfrac{1}{q}\langle I_\lambda'(u_n),u_n \rangle\\
&=\Big(\frac{1}{2}-\frac{1}{q}\Big)\| u_n\|^2
+\Big(\frac{1}{q}-\frac{1}{2^{**}}\Big)\|u_n\|_{2^{**}}^{2^{**}}
+\int_{\Omega}\Big(\frac{1}{q}f(x,u_n)u_n-F(x,u_n)\Big)\mathrm{d}x\\
&\geq\Big(\frac{1}{2}-\frac{1}{q}\Big)\| u_n\|^2,\qquad n\rightarrow\infty.
\end{align*}
When $q=2$, in accordance with $(f_2)$ and \eqref{Hhh}, we get, for $\lambda\in (0,\lambda_{s,2})$, that
\allowdisplaybreaks  \begin{align}
c+1+o(1)\|u_n\|
&\geq I_\lambda(u_n)-\frac{1}{\rho}\langle I_\lambda'(u_n),u_n \rangle\nonumber\\
&=\Big(\frac{1}{2}-\frac{1}{\rho}\Big)\int_{\Omega}(|\Delta u_n|^2-\lambda\frac{|u_n|^2}{|x|^s})\mathrm{d}x
+\Big(\frac{1}{\rho}-\frac{1}{2^{**}}\Big)\|u_n\|_{2^{**}}^{2^{**}}\nonumber\\
\label{Y3}& \ \ \ +\int_{\Omega}\Big(\frac{1}{\rho}f(x,u_n)u_n-F(x,u_n)\Big)\mathrm{d}x\\
&\geq\Big(\frac{1}{2}-\frac{1}{\rho}\Big)\int_{\Omega}(|\Delta u_n|^2-\lambda\frac{|u_n|^2}{|x|^s})\mathrm{d}x\nonumber\\
&\geq\Big(\frac{1}{2}-\frac{1}{\rho}\Big)\mu_s \|u_n\|^2,\qquad n\rightarrow\infty.\nonumber
\end{align}
It is obvious that in either case $\{u_n\}$ is a bounded sequence in $H_0^2(\Omega)$.
Consequently, by recalling Lemma \ref{lem1} one sees that there is a subsequence of $\{u_{n}\}$
(which we still denote by $\{u_{n}\}$) such that,  as $n\rightarrow\infty$,
\begin{eqnarray}\label{A}
\ \ \ \ \ \ \ \ \ \ \ \ \begin{cases}
u_{n}\rightharpoonup u \ in \ H_{0}^2(\Omega),\\
u_{n}\rightarrow u \ in \ L^r(\Omega)\ (1\leq r<2^{**}),\\
u_{n}\rightarrow u \ in \ H_0^1(\Omega),\\
u_{n}\rightarrow u \ in \ L^{q}(\Omega, |x|^{-s})\ (1\leq q<2^{**}(s)),\\
|u_{n}|^{2^{**}-2}u_n\rightharpoonup |u|^{2^{**}-2}u \ in \ L^{\frac{2^{**}}{2^{**}-1}}(\Omega),\\
u_{n}\rightarrow u \ \ a.e.\ in \ \Omega.
\end{cases}
\end{eqnarray}
In view of \eqref{Mm} (with $\varepsilon=1$), there exists a positive constant $C$, independent of $n$, such that
\begin{eqnarray*}
\begin{split}
|\int_\Omega f(x,u_n)u_n\mathrm{d}x|
&\leq\int_\Omega|f(x,u_n)||u_n|\mathrm{d}x\\
&\leq\|u_n\|_{2^{**}}^{2^{**}}+C_1(1)\|u_n\|_2^2\\
&\leq C.
\end{split}
\end{eqnarray*}
For any $\varepsilon>0$, take $\delta=\dfrac{\varepsilon}{C_1(\varepsilon)^{\frac{2^{**}}{2^{**}-2}}}$,
where $C_1(\varepsilon)>0$ is given in \eqref{Mm}.
Then for any measurable subset $E\subset\Omega$ with $mes\ E<\delta$, we obtain,
by recalling \eqref{Mm} again and applying H\"{o}lder's inequality, that
\allowdisplaybreaks
\begin{align*}
|\int_Ef(x,u_n)u_n\mathrm{d}x|
&\leq\varepsilon\int_E|u_n|^{2^{**}}\mathrm{d}x+C_1(\varepsilon)\int_E|u_n|^2\mathrm{d}x\\
&\leq\varepsilon \|u_n\|_{2^{**}}^{2^{**}}+C_1(\varepsilon)\|u_n\|_{2^{**}}^{2} (mesE)^{\frac{2^{**}-2}{2^{**}}}\\
&\leq C_1\varepsilon+C_2\varepsilon^{\frac{2^{**}-2}{2^{**}}},
\end{align*}
 uniformly  with respect to  $n\in \mathbb{N}$, where $C_1$, $C_2$ are positive constants  independent of $n$.
 Hence the family of functions $\{f(x,u_n)u_n\}$  is equi-absolutely-continuous.
In addition, recalling that $u_n\rightarrow u$ a.e. in $\Omega$ as $n\rightarrow\infty$ and $f(x,t)$ is continuous,
we get $f(x,u_n)u_n\rightarrow f(x,u)u$ a.e. in $\Omega$ as $n\rightarrow\infty$.  This, together with the fact that  $mes\ \Omega<\infty$, implies that $f(x,u_n)u_n\rightarrow f(x,u)u$  in measure.
Therefore, by virtue of Vitali convergence theorem, we get
\begin{eqnarray}\label{C}
\lim_{n\rightarrow\infty}\int_\Omega f(x,u_n)u_n\mathrm{d}x=\int_\Omega f(x,u)u\mathrm{d}x.
\end{eqnarray}
Similarly, we can prove that
\begin{eqnarray}\label{D}
\lim_{n\rightarrow\infty}\int_\Omega F(x,u_n)\mathrm{d}x=\int_\Omega F(x,u)\mathrm{d}x.
\end{eqnarray}
 By recalling \eqref{Mm} (with $\varepsilon=1$) and \eqref{A},  
  we have
\begin{eqnarray*}
\begin{split}
|\int_\Omega f(x,u_n)u\mathrm{d}x|
&\leq\int_\Omega|f(x,u_n)||u|\mathrm{d}x\\
&\leq
\int_\Omega|u_n|^{2^{**}-1}|u|\mathrm{d}x+C_1(1)\int_\Omega|u_n||u|\mathrm{d}x\\
&\rightarrow\|u\|_{2^{**}}^{2^{**}}+C_1(1)\|u\|_2^2, \qquad  n\rightarrow\infty,
\end{split}
\end{eqnarray*}
 which, together with the fact that $f(x,t)$ is a continuous function and $u_n\rightarrow u$ a.e. in $\Omega$ as $n\rightarrow\infty$, shows that $f(x,u_n)u\rightarrow f(x,u)u$ a.e. in $\Omega$ as $n\rightarrow\infty$.
Then, according to Lebesgue's dominated convergence theorem, we know
\begin{eqnarray}\label{E}
\lim_{n\rightarrow\infty}\int_\Omega f(x,u_n)u\mathrm{d}x=\int_\Omega f(x,u)u\mathrm{d}x.
\end{eqnarray}

To complete the proof, let $w_n=u_n-u$. Then $\{w_n\}$ is also a bounded sequence in $H_0^2(\Omega)$.
So there exists a subsequence of $\{w_{n}\}$ (which we still denoted by $\{w_{n}\}$) such that
\begin{equation}\label{limit-l}
\lim\limits_{n\rightarrow\infty}\|w_n\|^2=l\geq0.
\end{equation}
We claim that $l=0$.
Indeed, according to \eqref{A}, we have
\begin{eqnarray}\label{Zz}
\begin{split}
\|u_n\|^{2}&=\int_\Omega |\Delta w_n+\Delta u|^2\mathrm{d}x\\
&=\int_\Omega |\Delta w_n|^2\mathrm{d}x+\int_\Omega |\Delta u|^2\mathrm{d}x+2\int_\Omega \Delta w_n\Delta u\mathrm{d}x\\
&=\| w_n\|^2+\| u\|^{2}+o(1), \qquad  n\rightarrow\infty.
\end{split}
\end{eqnarray}
Moreover, since $\| u_n\|_{2^{**}}\leq C$ and $u_{n}\rightarrow u$  a.e. in $\Omega$ as $n\rightarrow\infty$,
one sees, by using Br\'{e}zis-Lieb's lemma, that
\begin{equation}\label{B}
\|u_n\|_{2^{**}}^{2^{**}}=\|w_n\|_{2^{**}}^{2^{**}}+\|u\|_{2^{**}}^{2^{**}}+o(1),\qquad n\rightarrow\infty.
\end{equation}
Therefore, in accordance with \eqref{A}, \eqref{C}, \eqref{E},  \eqref {Zz},  \eqref{B}  and the assumption that $I_\lambda'(u_{n})\rightarrow 0$ in $H^{-2}(\Omega)$ as $n\rightarrow\infty$, we have
\begin{eqnarray*}\label{}
\begin{split}
o(1)=&\langle I_\lambda'(u_n),u_n\rangle\\
=&\int_{\Omega}(|\Delta u_n|^2-\lambda\dfrac{|u_n|^q}{|x|^s})\mathrm{d}x
-\|u_n\|_{2^{**}}^{2^{**}}-\int_{\Omega}f(x,u_n)u_n\mathrm{d}x\\
=&\int_{\Omega}(|\Delta u|^2-\lambda\frac{|u|^q}{|x|^s})\mathrm{d}x
-\|u\|_{2^{**}}^{2^{**}}-\int_{\Omega}f(x,u)u\mathrm{d}x+\| w_n\|^2-\|w_n\|_{2^{**}}^{2^{**}}+o(1)\\
=&\langle I_\lambda'(u),u\rangle+\|w_n\|^2-\|w_n\|_{2^{**}}^{2^{**}}+o(1),\quad n\rightarrow\infty,
\end{split}
\end{eqnarray*}
and
\begin{eqnarray*}\label{}
\begin{split}
o(1)&=\langle I_\lambda'(u_n),u\rangle\\
&=\int_{\Omega}(\Delta u_n\Delta u-\lambda\frac{|u_n|^{q-2}u_nu}{|x|^s})\mathrm{d}x
-\int_{\Omega}|u_n|^{2^{**}-2}u_nu\mathrm{d}x-\int_{\Omega}f(x,u_n)u\mathrm{d}x\\
&=\int_{\Omega}(|\Delta u|^2-\lambda\frac{|u|^q}{|x|^s})\mathrm{d}x
-\|u\|_{2^{**}}^{2^{**}}-\int_{\Omega}f(x,u)u\mathrm{d}x+o(1)\\
&=\langle I_\lambda'(u),u\rangle+o(1), \quad n\rightarrow\infty.
\end{split}
\end{eqnarray*}
Combining the above two equalities one obtains
\begin{equation}\label{Y4}
\langle I_\lambda'(u),u\rangle=0,
\end{equation}
and
\begin{equation}\label{32}
\|w_n\|_{2^{**}}^{2^{**}}-\|w_n\|^2=o(1),\qquad  n\rightarrow\infty.
\end{equation}
 In addition, from the Sobolev embedding one has
 \begin{equation}\label{Vc}
  \|w_n\|_{2^{**}}^{2^{**}}\leq S^{-\frac{2^{**}}{2}}\| w_n\|^{2^{**}}<C,\qquad \forall\ n\in \mathbb{N}.
\end{equation}
It follows from \eqref{limit-l}, \eqref{32} and \eqref{Vc} that there is a subsequence of $\{w_{n}\}$ such that
\begin{equation}\label{limit-l2}
\lim\limits_{n\rightarrow\infty}\|w_n\|_{2^{**}}^{2^{**}}=\lim\limits_{n\rightarrow\infty}\|w_n\|^2=l.
\end{equation}
Letting $n\rightarrow\infty$ in \eqref{Vc}, we have $l\leq S^{-\frac{2^{**}}{2}}l^{\frac{2^{**}}{2}}$.
If $l>0$, then
\begin{equation}\label{L1}
l\geq S^{\frac{2^{**}}{2^{**}-2}}=S^{\frac{N}{4}}.
\end{equation}

On one hand, in view of \eqref{A}, \eqref{D}, \eqref {Zz}, \eqref{B} and the fact that $I_\lambda(u_n)=c+o(1)$ as $n\rightarrow\infty$,
we have
\begin{eqnarray*}
\begin{split}
o(1)+c=I_\lambda(u_n)
&=\frac{1}{2}\| u_n\|^2-\frac{\lambda}{q}\int_{\Omega}\frac{|u_n|^q}{|x|^s}\mathrm{d}x
-\frac{1}{2^{**}}\|u_n\|_{2^{**}}^{2^{**}}-\int_{\Omega}F(x,u_n)\mathrm{d}x\\
&=\frac{1}{2}\| u\|^2-\frac{\lambda}{q}\int_{\Omega}\frac{|u|^q}{|x|^s}\mathrm{d}x
-\frac{1}{2^{**}}\|u\|_{2^{**}}^{2^{**}}-\int_{\Omega}F(x,u)\mathrm{d}x\\
&\ \ \ +\frac{1}{2}\| w_n\|^2-\frac{1}{2^{**}}\|w_n\|_{2^{**}}^{2^{**}}+o(1)\\
&=I_\lambda(u)+\frac{1}{2}\| w_n\|^2-\frac{1}{2^{**}}\|w_n\|_{2^{**}}^{2^{**}}+o(1),\qquad  n\rightarrow\infty,
\end{split}
\end{eqnarray*}
which  yields that
$$I_\lambda(u)=c-\dfrac{1}{2}\| w_n\|^2+\dfrac{1}{2^{**}}\|w_n\|_{2^{**}}^{2^{**}}+o(1), \qquad n\rightarrow\infty.$$
Recalling \eqref{limit-l2} and \eqref{L1}, we get from the above equality that
\begin{eqnarray*}
\begin{split}
I_\lambda(u)=c-\Big(\frac{1}{2}-\frac{1}{2^{**}}\Big)l\leq c-\frac{2}{N}S^{\frac{N}{4}}<0.
\end{split}
\end{eqnarray*}

On the other hand, by \eqref{Y4} and $(f_2)$, we can derive
\allowdisplaybreaks  \begin{align}
I_\lambda(u)
&=I_\lambda(u)-\frac{1}{q}\langle I_\lambda'(u),u\rangle\nonumber\\
\label{Mk}&=\Big(\frac{1}{2}-\frac{1}{q}\Big)\| u\|^2+\Big(\frac{1}{q}-\frac{1}{2^{**}}\Big)\|u_n\|_{2^{**}}^{2^{**}}
+\int_{\Omega}\Big(\frac{1}{q}f(x,u)u-F(x,u)\Big)\mathrm{d}x\\
&\geq0,\nonumber
\end{align}
a contradiction. Thus, $\lim\limits_{n\rightarrow\infty}\|w_n\|^2=l=0$, which implies that $u_n\rightarrow u$ in $H_0^2(\Omega)$ as $n\rightarrow\infty$. The proof is complete.
\end{proof}

Before going further, we list some well-known estimates on the Talenti functions,
which will play a crucial role in estimating the mountain pass level of $I_\lambda$ around $0$.
For any $\varepsilon>0$, define
\begin{eqnarray*}
U_\varepsilon(x)=[N(N-4)(N^2-4)]^{\frac{N-4}{8}}\dfrac{\varepsilon^{\frac{N-4}{2}}}
{[\varepsilon^2+|x|^2]^{\frac{N-4}{2}}}, \ x \in \mathbb{R}^N. 
\end{eqnarray*}
Then $U_\varepsilon(x)$ is a solution of the critical problem
\begin{eqnarray*}
\Delta^2u=u^{2^{**}-1}, &x\in \mathbb{R}^N,\ \ \ N\geq5,
\end{eqnarray*}
and $\|U_\varepsilon\|^2=\|U_\varepsilon\|_{2^{**}}^{2^{**}}=S^{\frac{N}{4}}$,
where $S=\inf\limits_{u\in H_0^2(\Omega)\backslash\{0\}}\dfrac{\|u\|^2}{\|u\|_{2^{**}}^{2}}
=\dfrac{\|U_\varepsilon\|^2}{\|U_\varepsilon\|_{2^{**}}^{2}}$ is given in Remark \ref{remark1}.


The Talenti functions, after being truncated, are estimated in the following
(see \cite{li}, \cite{B. J. Xuan} and  \cite{4Hardylinj}).
\begin{lemma}\label{Y1}
Let $\tau\in C_0^\infty(\Omega)$ be a cut-off function such that $\tau(x)=\tau(|x|)$, $0\leq\tau(x)\leq1$ for $x\in \Omega$, and
\begin{eqnarray*}
\tau(x)=
\begin{cases}
1,&|x|<R,\\
0,&|x|> 2R,
\end{cases}
\end{eqnarray*}
where $R>0$ is a constant such that $B_{2R}(0)\subset\Omega$.
Set $u_\varepsilon(x)=\tau(x)U_\varepsilon(x)$. Suppose that $\varepsilon\rightarrow0$. Then
\begin{eqnarray*}\label{}
\begin{split}
&\|u_\varepsilon\|^2=S^{\frac{N}{4}}+O(\varepsilon^{N-4}),\\
&\|u_\varepsilon\|_{2^{**}}^{2^{**}}=S^{\frac{N}{4}}+O(\varepsilon^{N}).\\
\end{split}
\end{eqnarray*}
Set $v_\varepsilon(x)=\dfrac{u_\varepsilon}{\|u_\varepsilon\|_{2^{**}}}$. Then
\begin{eqnarray}\label{u}
\begin{split}
&\|v_\varepsilon\|^2=S+O(\varepsilon^{N-4}),\\
&\|v_\varepsilon\|_{2^{**}}^{2^{**}}=1,\\
\end{split}
\end{eqnarray}
\begin{eqnarray}\label{Nn}
\|v_\varepsilon\|_{\rho}^{\rho}=
\begin{cases}
O_1(\varepsilon^{\frac{(N-4)\rho}{2}}),&1<\rho<\frac{N}{N-4},\\
O_1(\varepsilon^{N-\frac{(N-4)\rho}{2}}|\ln\varepsilon|),&\rho=\frac{N}{N-4},\\
O_1(\varepsilon^{N-\frac{(N-4)\rho}{2}}),&\frac{N}{N-4}<\rho<2^{**},\\
\end{cases}
\end{eqnarray}
and
\begin{eqnarray}\label{Nnn}
\int_{\Omega}\dfrac{|v_\varepsilon|^q}{|x|^s}\mathrm{d}x=
\begin{cases}
O_1(\varepsilon^{\frac{(N-4)q}{2}}), & \ 1<q<\frac{N-s}{N-4},\\
O_1(\varepsilon^{N-\frac{(N-4)q}{2}-s}|\ln\varepsilon|), & \ q=\frac{N-s}{N-4},\\
O_1(\varepsilon^{N-\frac{(N-4)q}{2}-s}), & \ \frac{N-s}{N-4}<q<2^{**}(s).\\
\end{cases}
\end{eqnarray}
\end{lemma}

With the help of the Talenti functions given above, we can show that the mountain pass
level of $I_\lambda$ around $0$ is strictly less than $c(S)$.

\begin{lemma}\label{lem4}
Assume that $(f_1)$-$(f_2)$, and the condition  $(i)$ or $(ii)$ of Theorem \ref{th} hold.
Then there exists a $u^*\in H_0^2(\Omega)$ such that
\begin{eqnarray}\label{1}
\sup_{t\geq0}I_\lambda(tu^*)<c(S),
\end{eqnarray}
where $c(S):=\dfrac{2}{N}S^{\frac{N}{4}}$.
\end{lemma}

\begin{proof}
Define the fibering maps associated with the energy functional $I_\lambda$ by
\begin{eqnarray}\label{O}
\begin{split}
\psi_{u}(t)=I_\lambda(tu)&=\frac{1}{2}t^2\| u\|^2-\frac{1}{q}\lambda t^q\int_{\Omega}\frac{|u|^q}{|x|^s}\mathrm{d}x-\frac{1}{2^{**}}t^{2^{**}}\| u\|_{2^{**}}^{2^{**}}
-\int_{\Omega}F(x,tu)\mathrm{d}x,\ t\geq0.\\
\end{split}
\end{eqnarray}
Recalling \eqref{u} and the fact that $F(x,t)\geq0$ for any $x\in \Omega$ and $t\geq0$,
one sees, as $t\rightarrow0$, that
\begin{equation}\label{Ww}
\psi_{v_\varepsilon}(t)\leq\frac{1}{2}t^2\| v_\varepsilon\|^2\rightarrow0,
\end{equation}
uniformly for $\varepsilon\in(0,\varepsilon_1)$, where $\varepsilon_1>0$ is a
suitably small but fixed number and $v_\varepsilon$ is given in Lemma \ref{Y1}.
Therefore, there exists a $t_0>0$, independent of $\varepsilon$, such that
\begin{eqnarray}\label{t0}
\psi_{v_\varepsilon}(t)
<c(S),\qquad t\in(0,t_0).
\end{eqnarray}

 Set $g(t)=\dfrac{1}{2}t^2\| v_\varepsilon\|^2-\dfrac{1}{2^{**}}t^{2^{**}}$, then
\begin{eqnarray*}\label{}
\psi_{v_\varepsilon}(t)=g(t)-\frac{1}{q}\lambda t^q\int_{\Omega}\frac{|v_\varepsilon|^q}{|x|^s}\mathrm{d}x
-\int_{\Omega}F(x,tv_\varepsilon)\mathrm{d}x,\qquad  t\geq0.
\end{eqnarray*}
According to \eqref{Mmmm},  there exists a positive constant $C$ such that
\begin{eqnarray}\label{Nb}
\begin{split}
\psi_{v_\varepsilon}(t)
&\leq \max\limits_{t\geq t_0}g(t)-\dfrac{1}{q}\lambda t^q\int_{\Omega}\frac{|v_\varepsilon|^q}{|x|^s}\mathrm{d}x
-Ct^\rho\|v_\varepsilon\|_\rho^{\rho}\\
&\leq
\max\limits_{t\geq 0}g(t)-\dfrac{1}{q}\lambda t_0^q\int_{\Omega}\frac{|v_\varepsilon|^q}{|x|^s}\mathrm{d}x
-Ct_0^\rho\|v_\varepsilon\|_\rho^{\rho}, \qquad t\in[t_0,\infty).\\
\end{split}
\end{eqnarray}
By a direct calculation,   $g$ takes its maximum at $t_\varepsilon^*:=\|v_\varepsilon\|^{\frac{2}{2^{**}-2}}$ and $g(t_\varepsilon^*)=\dfrac{2}{N}\|v_\varepsilon\|^{\frac{N}{2}}$.
Therefore, in view of this, \eqref{u} and \eqref{Nb}, one sees, for $t\geq t_0$, that
\allowdisplaybreaks  \begin{align*}
\psi_{v_\varepsilon}(t)
&\leq g(t_\varepsilon^*)-\frac{1}{q}\lambda t_0^q\int_{\Omega}\frac{|v_\varepsilon|^q}{|x|^s}\mathrm{d}x
-Ct_0^\rho\|v_\varepsilon\|_\rho^{\rho}\\
&=\frac{2}{N}\big(S+O(\varepsilon^{N-4})\big)^{\frac{N}{4}}
-\frac{1}{q}\lambda t_0^q\int_{\Omega}\frac{|v_\varepsilon|^q}{|x|^s}\mathrm{d}x
-Ct_0^\rho\|v_\varepsilon\|_\rho^{\rho}\\
&= c(S)+O(\varepsilon^{N-4})
-\frac{1}{q}\lambda t_0^q\int_{\Omega}\frac{|v_\varepsilon|^q}{|x|^s}\mathrm{d}x
-Ct_0^\rho\|v_\varepsilon\|_\rho^{\rho},\qquad \varepsilon\rightarrow0.
\end{align*}
Our goal is to show, for $\varepsilon$ suitably small, that
\begin{equation}\label{T0}
\psi_{v_\varepsilon}(t)<c(S),\qquad t\in[t_0,\infty),
\end{equation}
 which is fulfilled once we can prove that
\begin{eqnarray}\label{Y2}
O(\varepsilon^{N-4})
-\frac{1}{q}\lambda t_0^q\int_{\Omega}\frac{|v_\varepsilon|^q}{|x|^s}\mathrm{d}x
-Ct_0^\rho\|v_\varepsilon\|_\rho^{\rho}<0.
\end{eqnarray}
It is easy to see that \eqref{Y2} holds if either $(I)$ or $(II)$ of the following is valid
\begin{eqnarray*}
\begin{split}
&(I)\ \  O(\varepsilon^{N-4})
-Ct_0^\rho\|v_\varepsilon\|_\rho^{\rho}<0;\\
&(II)\ \ O(\varepsilon^{N-4})
-\frac{1}{q}\lambda t_0^q\int_{\Omega}\frac{|v_\varepsilon|^q}{|x|^s}\mathrm{d}x
<0.
\end{split}
\end{eqnarray*}

First, when $0\leq s<4$, $q\geq 2$
and $\max\{\frac{N}{N-4},\frac{8}{N-4}\}<\rho<2^{**}$, by recalling \eqref{Nn},
we obtain
\begin{eqnarray*}
 \lim\limits_{\varepsilon\rightarrow0}\dfrac{\varepsilon^{N-4}}
{\varepsilon^{N-\frac{(N-4)\rho}{2}}}= \lim\limits_{\varepsilon\rightarrow0}\varepsilon^{\frac{(N-4)\rho}{2}-4}=0,
\end{eqnarray*}
which implies that $(I)$ holds.

Next, we shall show that $(II)$ is fulfilled for small $\varepsilon$
when other cases in (i) or (ii) of Theorem \ref{th} are satisfied, with the help of \eqref{Nnn}.
Indeed, if $q=2$ and $N=8-s$, then $q=\frac{N-s}{N-4}$, which, together with \eqref{Nnn}, implies that
\begin{eqnarray}\label{00}
 \lim\limits_{\varepsilon\rightarrow0}\dfrac{\varepsilon^{N-4}}
{\varepsilon^{N-\frac{(N-4)q}{2}-s}|\ln\varepsilon|}= \lim\limits_{\varepsilon\rightarrow0}\dfrac{1}
{|\ln\varepsilon|}=0.
\end{eqnarray}
If $q=2$ and $N>8-s$, then $q>\frac{N-s}{N-4}$ and
\begin{eqnarray}\label{II}
 \lim\limits_{\varepsilon\rightarrow0}\dfrac{\varepsilon^{N-4}}
{\varepsilon^{N-\frac{(N-4)q}{2}-s}}= \lim\limits_{\varepsilon\rightarrow0}
\varepsilon^{N-(8-s)}=0.
\end{eqnarray}
It follows from \eqref{00} and \eqref{II} that $(II)$ is valid if $q=2$ and $N\geq8-s$.

Similarly, if $q>2$ and $\max\{\frac{N-s}{N-4},\frac{2(4-s)}{N-4}\}< q<2^{**}(s)$,
we have
\begin{eqnarray*}
 \lim\limits_{\varepsilon\rightarrow0}\dfrac{\varepsilon^{N-4}}
{\varepsilon^{N-\frac{(N-4)q}{2}-s}}= \lim\limits_{\varepsilon\rightarrow0}
\varepsilon^{\frac{1}{2}[(N-4)q-(8-2s)]}=0,
\end{eqnarray*}
which also implies that $(II)$ holds.

In conclusion, 
 if the condition  $(i)$ or $(ii)$ of Theorem \ref{th} holds, by combining \eqref{t0} with \eqref{T0}, we have, for $\varepsilon$ suitably small, that
\begin{eqnarray*}\label{}
\sup_{t\geq 0}\psi_{v_\varepsilon}(t)
<c(S).
\end{eqnarray*}
Fix such an $\varepsilon>0$ and take $u^*\equiv v_\varepsilon$. The proof is complete.
\end{proof}

In what follows, we shall show the existence of the weak solutions to problem \eqref{P1.1}
on the basis of Lemmas \ref{lem3} and \ref{lem4} and the Mountain Pass Lemma.

\begin{proof}[Proof of Theorem \ref{th} $(i)$ and $(ii)$]
We first show that the functional $I_\lambda$ satisfies the mountain pass geometry in both cases.
If $q>2$, according to  \eqref{Mmm} and \eqref{Vv}, for any $\varepsilon>0$, there exists a $C_3(\varepsilon)>0$ such that
\begin{eqnarray}\label{mp-1}
\begin{split}
I_\lambda(u)
&\geq\frac{1}{2}\|u\|^2-\frac{\lambda_{s,q}^{-\frac{q}{2}}}{q}\lambda\|u\|^q
-\frac{S^{-\frac{2^{**}}{2}}}{2^{**}}\|u\|^{2^{**}}
-\frac{\lambda_{0,2}^{-1}}{2}\varepsilon\|u\|^{2}
-C_3(\varepsilon)S^{-\frac{2^{**}}{2}}\|u\|^{2^{**}}\\
&=(1-\lambda_{0,2}^{-1}\varepsilon)\frac{\|u\|^2}{2}
-\frac{\lambda_{s,q}^{-\frac{q}{2}}}{q}\lambda\|u\|^q
-\Big(\frac{1}{{2^{**}}}+C_3(\varepsilon)\Big)S^{-\frac{2^{**}}{2}}\|u\|^{2^{**}},
\qquad \forall \ u\in H_{0}^{2}(\Omega)\backslash\{0\}.
\end{split}
\end{eqnarray}
If $q=2$, by using  \eqref{Mmm}, \eqref{Vv} again and \eqref{Hhh}, we obtain that
\begin{eqnarray}\label{7.1}
\begin{split}
I_\lambda(u)
&\geq\frac{\mu_s}{2}\|u\|^2
-\frac{S^{-\frac{2^{**}}{2}}}{2^{**}}\|u\|^{2^{**}}
-\frac{\lambda_{0,2}^{-1}}{2}\varepsilon\|u\|^{2}
-C_3(\varepsilon)S^{-\frac{2^{**}}{2}}\|u\|^{2^{**}}\\
&=(\mu_s-\lambda_{0,2}^{-1}\varepsilon)\frac{\|u\|^2}{2}
-\Big(\frac{1}{{2^{**}}}+C_3(\varepsilon)\Big)S^{-\frac{2^{**}}{2}}\|u\|^{2^{**}},\qquad \forall \ u\in H_{0}^{2}(\Omega)\backslash\{0\}.
\end{split}
\end{eqnarray}
Choosing $\varepsilon>0$ so small that $1-\lambda_{0,2}^{-1}\varepsilon>0$ and $\mu_s-\lambda_{0,2}^{-1}\varepsilon>0$,
and noticing $2\leq q<2^{**}$, one sees that there exist $\beta,\ r>0$ such that $I_\lambda(u)\geq\beta$ for all $\|u\|=r$.

On the other hand,
recalling that $F(x, t)\geq0$  for any $x\in \Omega$ and $t\geq0$ due to $(f_2)$, we have,
for any $u\in H^2_0(\Omega)\backslash\{0\}$
\begin{eqnarray}\label{7.2}
\begin{split}
\psi_{u}(t)
\leq\frac{1}{2}t^2\|u\|^2-\frac{1}{2^{**}}t^{2^{**}}\|u\|_{2^{**}}^{2^{**}},\\
\end{split}
\end{eqnarray}
which implies that $\lim\limits_{t\rightarrow\infty}\psi_{u}(t)=-\infty$. Therefore, there exists a $t_u>0$ suitably large  such that $\|t_uu\|>r$ and $\psi_{u}(t_u)=I_\lambda(t_uu)<0$. Thus, $I_\lambda$ satisfies the mountain pass geometry around $0$,
and there exists a sequence $\{u_n\}\subset H_0^2(\Omega)$ such that $I_\lambda(u_{n})\rightarrow c_0\geq\beta$ and $I_\lambda'(u_{n})\rightarrow 0$ in $H^{-2}(\Omega)$ as $n\rightarrow\infty$,
where
\begin{equation*}
c_0=\inf_{\gamma\in\Gamma}\max_{t\in[0,1]}I_\lambda(\gamma(t))\
and\ \
\Gamma=\{\gamma\in C([0,1],H_{0}^{2}(\Omega)): \gamma(0)=0, \gamma(1)=t_{u^*}u^*\},
\end{equation*}
and $u^*$ is given in Lemma \ref{lem4}.
In view of \eqref{mp-1} and \eqref{1}, one sees
\begin{equation}\label{BM}
 c_0
\leq\max\limits_{t\in[0,1]}I_\lambda(tt_{u^*}u^*)
\leq\sup\limits_{t\geq0}I_\lambda(tu^*)<c(S).
\end{equation}
It then follows  from Lemma \ref{lem3} that $I_\lambda(u)$  satisfies the $(PS)_{c_0}$ condition.
Consequently, there exists  a convergent subsequence of $\{u_{n}\}$, still denoted by $\{u_{n}\}$,
such that $u_{n}\rightarrow u$ in $H_0^2(\Omega)$ as $n\rightarrow\infty$,
which implies that $I_\lambda(u)=c_0$ and $I_\lambda'(u)=0$, i.e.,
$u$ is a nonnegative  solution to problem \eqref{P1.1}.
The proof of (i) and (ii) of Theorem \ref{th} is complete.
\end{proof}

\subsection{$The \ case \ s=4.$}

When $s=4$, the difficulty we encounter is the lack of compactness of
the mapping $u\rightarrow\dfrac{u}{|x|^{2}}$ from $H_0^2(\Omega)$ into
$L^2(\Omega)$ and the Sobolev embedding $H_0^2(\Omega)\hookrightarrow L^{2^{**}}(\Omega)$,
which prevents us from establishing the usual (PS) condition directly.
However, inspired by some ideas from \cite{li}, we can show that
$I_\lambda$ satisfies the $(PS)_c$ condition for some $c$,
and then the existence of a mountain pass type solution to problem \eqref{P1.1} follows.

\begin{lemma}\label{lem2.1}
Assume that $q=2$, $s=4$ and that $f(x,t)$ satisfies $(f_1)$ and $(f_2)$.
Let $\{u_{n}\}\subset H_0^2(\Omega)$ be a  sequence such that $I_\lambda(u_{n})\rightarrow c<c(S_\lambda)$
and $I_\lambda'(u_{n})\rightarrow 0$ in $H^{-2}(\Omega)$ as $n\rightarrow\infty$,
where $c(S_\lambda):=\dfrac{2}{N}S_\lambda^{\frac{N}{4}}$. Then $I_\lambda(u)$ satisfies the $(PS)_c$
condition provided that  $0<\lambda<\lambda_{4,2}$.
\end{lemma}

\begin{proof}
Recalling \eqref{Y3} with $s=4$, one sees that $\{u_n\}$ is bounded
in $H_0^2(\Omega)$ when $0<\lambda<\lambda_{4,2}$.
Hence, according to  Lemma \ref{lem1}, there is a subsequence of $\{u_{n}\}$,
still denoted by $\{u_{n}\}$, such that, as $n\rightarrow\infty$,
\begin{eqnarray}\label{A1}
\ \ \ \ \ \ \ \ \ \ \ \ \begin{cases}
u_{n}\rightharpoonup u \ in \ H_{0}^2(\Omega),\\
u_{n}\rightarrow u \ in \ L^r(\Omega)\ (1\leq r<2^{**}),\\
u_{n}\rightarrow u \ in \ H_0^1(\Omega),\\
u_{n}\rightharpoonup u \ in \ L^{2}(\Omega, |x|^{-2}),\\
|u_{n}|^{2^{**}-2}u_n\rightharpoonup |u|^{2^{**}-2}u \ in \ L^{\frac{2^{**}}{2^{**}-1}}(\Omega),\\
u_{n}\rightarrow u \ a.e.\ in \ \Omega.
\end{cases}
\end{eqnarray}

As was done in the proof of Lemma \ref{lem3}, set $w_n=u_n-u$.
Then $\{w_n\}$ is bounded in $H_0^2(\Omega)$,
which ensures that there exists a subsequence of $\{w_{n}\}$,
still denoted by $\{w_{n}\}$, such that
$$\lim\limits_{n\rightarrow\infty}\|w_n\|^2=l\geq0.$$
In what follows, we shall show $l=0$. In view of \eqref{A1}
and Br\'{e}zis-Lieb's lemma, we know that \eqref{Zz} and \eqref{B} are still valid.
In addition, from Lemma  \ref{lemB}
one see,  as $n\rightarrow\infty$, that
\begin{eqnarray}\label{B2}
\begin{split}
\int_{\Omega}\frac{|u_n|^2}{|x|^4}\mathrm{d}x&=
\int_{\Omega}\frac{|w_n|^2}{|x|^4}\mathrm{d}x
+\int_{\Omega}\frac{|u|^2}{|x|^4}\mathrm{d}x+o(1).
\end{split}
\end{eqnarray}
Combining \eqref{C}, \eqref{E}, \eqref{Zz}, \eqref{B}, \eqref{A1},
 \eqref{B2} with the assumption that $I_\lambda'(u_{n})\rightarrow 0$
 in $H^{-2}(\Omega)$ as $n\rightarrow\infty$, one gets
\begin{eqnarray*}
\begin{split}
o(1)=\langle I_\lambda'(u_n),u_n\rangle
&=\int_{\Omega}(|\Delta u_n|^2-\lambda\frac{|u_n|^2}{|x|^4})\mathrm{d}x
-\|u_n\|_{2^{**}}^{2^{**}}-\int_{\Omega}f(x,u_n)u_n\mathrm{d}x\\
&=\int_{\Omega}(|\Delta u|^2-\lambda\frac{|u|^2}{|x|^4})\mathrm{d}x
-\|u\|_{2^{**}}^{2^{**}}-\int_{\Omega}f(x,u)u\mathrm{d}x\\
& \ \ \ +\| w_n\|^2-\lambda\int_{\Omega}\frac{|w_n|^2}{|x|^4}\mathrm{d}x
-\|w_n\|_{2^{**}}^{2^{**}}+o(1)\\
&=\langle I_\lambda'(u),u\rangle+\| w_n\|^2-\lambda\int_{\Omega}\frac{|w_n|^2}{|x|^4}\mathrm{d}x
-\|w_n\|_{2^{**}}^{2^{**}}+o(1),\qquad n\rightarrow\infty,
\end{split}
\end{eqnarray*}
and
\allowdisplaybreaks  \begin{align*}
o(1)=\langle I_\lambda'(u_n),u\rangle
&=\int_{\Omega}(\Delta u_n\Delta u-\lambda\frac{|u_n|u}{|x|^4})\mathrm{d}x
-\int_{\Omega}|u_n|^{2^{**}-2}u_nu\mathrm{d}x
-\int_{\Omega}f(x,u_n)u\mathrm{d}x\\
&=\int_{\Omega}(|\Delta u|^2-\lambda\frac{|u|^2}{|x|^4})\mathrm{d}x
-\|u\|_{2^{**}}^{2^{**}}-\int_{\Omega}f(x,u)u\mathrm{d}x+o(1)\\
&=\langle I_\lambda'(u),u\rangle+o(1), \qquad n\rightarrow\infty.
\end{align*}
According to the above two equalities, we obtain
\eqref{Y4}
and
\begin{equation}\label{33}
\| w_n\|^2-\lambda\int_{\Omega}\frac{|w_n|^2}{|x|^4}\mathrm{d}x
-\|w_n\|_{2^{**}}^{2^{**}}=o(1), \qquad n\rightarrow\infty.
\end{equation}
In addition, in view of  \eqref{Hh}, one has
\begin{equation}\label{31}
\|w_n\|_{2^{**}}^{2^{**}}\leq S_\lambda^{-\frac{2^{**}}{2}}(\|w_n\|^2
-\lambda\int_{\Omega}\frac{|w_n|^2}{|x|^4}\mathrm{d}x)^{\frac{2^{**}}{2}}\leq
S_\lambda^{-\frac{2^{**}}{2}}\|w_n\|^2<C.
\end{equation}
It then follows from \eqref{33} and \eqref{31} that we can take a subsequence of $\{w_n\}$ such that
\begin{equation}\label{Ff}
\lim\limits_{n\rightarrow\infty}(\|w_n\|^2
-\lambda\displaystyle{\int_{\Omega}\frac{|w_n|^2}{|x|^4}\mathrm{d}x})
=\lim\limits_{n\rightarrow\infty}\|w_n\|_{2^{**}}^{2^{**}}=k\geq0.
\end{equation}
Letting $n\rightarrow\infty$ in \eqref{31}, we obtain that $k\leq S_\lambda^{-\frac{2^{**}}{2}}k^{\frac{2^{**}}{2}}$.

If $k>0$, then
\begin{equation}\label{Y6}
k\geq S_\lambda^{\frac{2^{**}}{2^{**}-2}}=S_\lambda^{\frac{N}{4}}.
\end{equation}
In accordance with \eqref{D}, \eqref{Zz}, \eqref{B}, \eqref{A1},  \eqref{B2} and $I_\lambda(u_n)=c+o(1)\ as \ n\rightarrow\infty$, we have
\begin{eqnarray*}
\begin{split}
o(1)+c=I_\lambda(u_n)&=\frac{1}{2}\int_{\Omega}(|\Delta u_n|^2-\lambda\frac{|u_n|^2}{|x|^4})\mathrm{d}x
-\frac{1}{2^{**}}\|u_n\|_{2^{**}}^{2^{**}}
-\int_{\Omega}F(x,u_n)\mathrm{d}x\\
&=\frac{1}{2}\int_{\Omega}(|\Delta u|^2-\lambda\frac{|u|^2}{|x|^4})\mathrm{d}x-\frac{1}{2^{**}}\|u\|_{2^{**}}^{2^{**}}
-\int_{\Omega}F(x,u)\mathrm{d}x\\
& \ \ \ +\frac{1}{2}\| w_n\|^2-\frac{1}{2}\lambda\int_{\Omega}\frac{|w_n|^2}{|x|^4}\mathrm{d}x
-\frac{1}{2^{**}}\|w_n\|_{2^{**}}^{2^{**}}+o(1)\\
&=I_\lambda(u)+\frac{1}{2}\| w_n\|^2-\frac{1}{2}\lambda\int_{\Omega}\frac{|w_n|^2}{|x|^4}\mathrm{d}x
-\frac{1}{2^{**}}\|w_n\|_{2^{**}}^{2^{**}}+o(1),\qquad n\rightarrow\infty,
\end{split}
\end{eqnarray*}
which then yields that
\begin{eqnarray*}
I_\lambda(u)=c-\Big(\dfrac{1}{2}\| w_n\|^2-\dfrac{\lambda}{2}\displaystyle{\int_{\Omega}\dfrac{|w_n|^2}{|x|^4}\mathrm{d}x}
-\dfrac{1}{2^{**}}\|w_n\|_{2^{**}}^{2^{**}}\Big)+o(1),\qquad  n\rightarrow\infty.
\end{eqnarray*}
From the above equality and  recalling  \eqref{Ff} and \eqref{Y6}, 
we have
\begin{eqnarray*}\label{}
\begin{split}
I_\lambda(u)=c-\Big(\frac{1}{2}-\frac{1}{2^{**}}\Big)k\leq c
-\frac{2}{N}S_\lambda^{\frac{N}{4}}<0,
\end{split}
\end{eqnarray*}
which contradicts  \eqref{Mk}.
Thus, $k=0$. Noticing that $0<\lambda<\lambda_{4,2}$, we obtain from \eqref{Hhh} that
there exists a $\mu_4>0$ such that
\begin{eqnarray*}
\mu_4\| w_n\|^2\leq\int_{\Omega}(|\Delta w_n|^2-\lambda\frac{|w_n|^2}{|x|^4})\mathrm{d}x=o(1),\qquad n\rightarrow\infty,
\end{eqnarray*}
which implies $\lim\limits_{n\rightarrow\infty}\|w_n\|^2=l=0$,
i.e., $u_n\rightarrow u$ in $H_0^2(\Omega)$ as $n\rightarrow\infty$. The proof is complete.
\end{proof}

The following lemma, which shows that the mountain pass level of $I_\lambda(u)$
around $0$ is strictly less than $c(S_\lambda)$, is parallel to Lemma \ref{lem4}.

\begin{lemma}\label{r}
Assume that $(f_1)-(f_2)$ and the condition $(iii)$ of Theorem \ref{th} hold.
Then there exists a $\widetilde{u}_\lambda^*\in H_0^2(\Omega)$ such that
\begin{eqnarray}\label{F1}
\sup_{t\geq0}I_\lambda(t\widetilde{u}_\lambda^*)<c(S_\lambda),
\end{eqnarray}
where $c(S_\lambda):=\dfrac{2}{N}S_\lambda^{\frac{N}{4}}$.
\end{lemma}

\begin{proof}
To estimate the  mountain pass level of $I_\lambda$ around $0$, we introduce the Talenti functions.
For any $\varepsilon>0$, define
\begin{eqnarray*}
\widetilde{U}_{\varepsilon,\lambda}(x)=\varepsilon^{2-\frac{N}{2}
}\widetilde{U}_\lambda\displaystyle{(\dfrac{x}{\varepsilon})},\qquad \lambda\in[0,\lambda_{4,2}),
\end{eqnarray*}
  where $\widetilde{U}_\lambda(x)>0$ is radially symmetric and is a solution to the critical problem
\begin{eqnarray*}
\Delta^2u-\lambda\dfrac{u}{|x|^4}=u^{2^{**}-1}, &x\in \mathbb{R}^N,\ \ \ N\geq5.
\end{eqnarray*}
Then, according to Theorem 1.1 of \cite{D2}, $\widetilde{U}_{\varepsilon,\lambda}(x)$ is an
extremal function of $S_\lambda$ and
\begin{eqnarray*}
\int_{\mathbb{R}^N}\left(|\Delta \widetilde{U}_{\varepsilon,\lambda}(x)|^2
-\lambda\frac{|\widetilde{U}_{\varepsilon,\lambda}
(x)|^2}{|x|^4}\right)\mathrm{d}x=\displaystyle{\int_{\mathbb{R}^N}|\widetilde{U}
_{\varepsilon,\lambda}(x)|^{2^{**}}\mathrm{d}x}
=S_\lambda^{\frac{N}{4}},
\end{eqnarray*}
where $S_\lambda$ is given in Remark \ref{SS}. Set $r=|x|$, then one also has
\begin{eqnarray*}
\begin{split}
\widetilde{U}_\lambda(r)&= O_1(r^{-\gamma_\lambda}),\ \ \widetilde{U}_\lambda'(r)= O_1(r^{-\gamma_\lambda-1}),\qquad r\rightarrow+\infty,\\
\widetilde{U}_\lambda(r)&= O_1(r^{-\widetilde{\gamma}_\lambda}), \qquad r\rightarrow 0,
\end{split}
\end{eqnarray*}
where $\gamma_\lambda=(N-4)(1-\dfrac{\alpha}{2})$, $\widetilde{\gamma}_\lambda=(N-4)\dfrac{\alpha}{2}$, and
\begin{eqnarray*}
 \alpha:=\varphi(\lambda)=1-\frac{\sqrt{N^2-4N+8-4\sqrt{(N-2)^2+\lambda}}}{N-4},
\qquad  \lambda\in(0,\lambda_{4,2}).
 \end{eqnarray*}
It is directly verifies that $\varphi : [0, \lambda_{4,2}] \rightarrow [0, 1]$ is continuous and strictly increasing.
Therefore, $\alpha\in(0,1)$ and $0<\widetilde{\gamma}_\lambda<\frac{N-4}{2}< \gamma_\lambda< N-4$ for $\lambda\in(0,\lambda_{4,2})$.

Let $\tau(x)$ be given in Lemma \ref{Y1} and set $\widetilde{u}_{\varepsilon,\lambda}(x)=\tau(x)\widetilde{U}_{\varepsilon,\lambda}(x)$.
Then as $\varepsilon\rightarrow0$, we have (see \cite{D2,li})
\begin{eqnarray*}
\begin{split}
\int_{\Omega}\left(|\Delta \widetilde{u}_{\varepsilon,\lambda}(x)|^2
-\lambda\frac{|\widetilde{u}_{\varepsilon,\lambda}(x)|^2}{|x|^4}\right)\mathrm{d}x
&=
S_\lambda^{\frac{N}{4}}
+O(\varepsilon^{2(\gamma_\lambda-\frac{N-4}{2})}),\\
\|\widetilde{u}_{\varepsilon,\lambda}\|_{2^{**}}^{2^{**}}&=S_\lambda^{\frac{N}{4}}
+O(\varepsilon^{2^{**}\gamma_\lambda-N}),\\
\end{split}
\end{eqnarray*}
and
\begin{eqnarray}\label{Jj}
\|\widetilde{u}_{\varepsilon,\lambda}\|_\rho^\rho&=
\begin{cases}
O_1(\varepsilon^{\rho(\gamma_\lambda-\frac{N-4}{2})}),
&1\leq\rho<\frac{N}{\gamma_\lambda},\\
O_1(\varepsilon^{N-\frac{N-4}{2}\rho}|\ln\varepsilon|),
&\rho=\frac{N}{\gamma_\lambda},\\
O_1(\varepsilon^{N-\frac{N-4}{2}\rho}),
&\frac{N}{\gamma_\lambda}<\rho<2^{**}.\\
\end{cases}
\end{eqnarray}


As was done in the proof of Lemma \ref{lem4}, we consider the fibering maps $\psi_{u}(t)$ given in \eqref{O} with $u=\widetilde{u}_{\varepsilon,\lambda}$. Since \eqref{Ww} still holds for $\psi_{\widetilde{u}_{\varepsilon,\lambda}}(t)$,
there exists a $\widetilde{t}_0>0$, independent of $\varepsilon$,  such that
\begin{eqnarray}\label{t00}
\psi_{\widetilde{u}_{\varepsilon,\lambda}}(t)
<c(S_\lambda),\qquad t\in(0,\widetilde{t}_0).
\end{eqnarray}

Next, set
$$\widetilde{g}(t)=\dfrac{1}{2}t^2\displaystyle{\int_{\Omega}\left(|\Delta \widetilde{u}_{\varepsilon,\lambda}|^2
-\lambda\dfrac{|\widetilde{u}_{\varepsilon,\lambda}|^2}{|x|^4}\right)\mathrm{d}x}
-\dfrac{1}{2^{**}}t^{2^{**}}\|\widetilde{u}_{\varepsilon,\lambda}\|_{2^{**}}^{2^{**}},$$
then
\begin{eqnarray*}\label{}
\psi_{\widetilde{u}_{\varepsilon,\lambda}}(t)=\widetilde{g}(t)
-\displaystyle{\int_{\Omega}
F(x,t\widetilde{u}_{\varepsilon,\lambda})\mathrm{d}x}.
\end{eqnarray*}
On recalling \eqref{Mmmm} one sees that there exists a positive constant $C$ such that
\begin{eqnarray}\label{Lp1}
\begin{split}
\psi_{\widetilde{u}_{\varepsilon,\lambda}}(t)
&\leq \max\limits_{t\geq \widetilde{t}_0}\widetilde{g}(t)-\int_{\Omega}F(x,t \widetilde{u}_{\varepsilon,\lambda})\mathrm{d}x\\
&\leq
\max\limits_{t\geq 0}\widetilde{g}(t)
-C\widetilde{t}_0^\rho\|\widetilde{u}_{\varepsilon,\lambda}\|_\rho^{\rho}, \qquad t\in[\widetilde{t}_0,\infty).\\
\end{split}
\end{eqnarray}
Direct computation shows that $\widetilde{g}$ attains its maximum at
\begin{eqnarray}\label{Lp2}
\widetilde{t}_\varepsilon^*:= \left(\dfrac{\int_{\Omega}(|\Delta \widetilde{u}_{\varepsilon,\lambda}|^2
-\lambda\frac{|\widetilde{u}_{\varepsilon,\lambda}
|^2}{|x|^4})\mathrm{d}x}{\|\widetilde{u}_{\varepsilon,\lambda}
\|_{2^{**}}^{2^{**}}}\right)^{\frac{1}{2^{**}-2}}
.
\end{eqnarray}
Therefore, by \eqref{Lp1} and \eqref{Lp2}, one sees, for $t\in[\widetilde{t}_0,\infty)$, that
\begin{eqnarray*}
\begin{split}
\psi_{\widetilde{u}_{\varepsilon,\lambda}}(t)
&\leq \widetilde{g}(\widetilde{t}_\varepsilon^*)
-C\widetilde{t}_0^\rho\|\widetilde{u}_{\varepsilon,\lambda}\|_\rho^{\rho}\\
&=\Big(\frac{1}{2}-\frac{1}{2^{**}}\Big)
\frac{\big(\int_{\Omega}(|\Delta \widetilde{u}_{\varepsilon,\lambda}|^2-\lambda\frac{|\widetilde{u}
_{\varepsilon,\lambda}|^2}{|x|^4})\mathrm{d}x\big)^{\frac{N}{4}}}
{\|\widetilde{u}_{\varepsilon,\lambda}\|_{2^{**}}^{\frac{(N-4)2^{**}}{4}}}
-C\widetilde{t}_0^\rho\|\widetilde{u}_{\varepsilon,\lambda}\|_\rho^{\rho}\\
&=\frac{2}{N}
\big(S_\lambda^{\frac{N}{4}}
+O(\varepsilon^{2(\gamma_\lambda-\frac{N-4}{2})})\big)^{\frac{N}{4}}\big(S_\lambda^{\frac{N}{4}}
+O(\varepsilon^{2^{**}\gamma_\lambda-N})\big)^{\frac{4-N}{4}}
-C\widetilde{t}_0^\rho\|\widetilde{u}_{\varepsilon,\lambda}\|_\rho^{\rho}\\
&=c(S_\lambda)
+O(\varepsilon^{2(\gamma_\lambda-\frac{N-4}{2})})
-C\widetilde{t}_0^\rho\|\widetilde{u}_{\varepsilon,\lambda}\|_\rho^{\rho}, \qquad as\ \varepsilon\rightarrow0,
\end{split}
\end{eqnarray*}
where in the last equality we have used the fact that
$$(2^{**}\gamma_\lambda-N)-2(\gamma_\lambda-\frac{N-4}{2})=4-4\alpha\geq0.$$

To complete the proof of this lemma, it remains to show, for $\varepsilon$ suitably small,  that
\begin{eqnarray}\label{Wd}
\begin{split}
\psi_{\widetilde{u}_{\varepsilon,\lambda}}(t)
&< c(S_\lambda),  \qquad t\in[\widetilde{t}_0,\infty),
\end{split}
\end{eqnarray}
which is fulfilled once we can prove that
\begin{eqnarray}\label{MM}
O(\varepsilon^{2(\gamma_\lambda-\frac{N-4}{2})})
-C\widetilde{t}_0^\rho\|\widetilde{u}_{\varepsilon,\lambda}\|_\rho^{\rho}< 0.
\end{eqnarray}

In view of \eqref{Jj},  if $\rho>\max\{\frac{N}{\gamma_\lambda},
\frac{4(N-2-\gamma_\lambda)}{N-4}\}$, one  sees that 
\allowdisplaybreaks  \begin{align*}
\lim\limits_{\varepsilon\rightarrow0}
 \dfrac{\varepsilon^{2(\gamma_\lambda-\frac{N-4}{2})}}
 {\varepsilon^{N-\frac{N-4}{2}\rho}}=\lim\limits_{\varepsilon\rightarrow0}
 \varepsilon^{\frac{1}{2}[\rho(N-4)-4(N-2-\gamma_\lambda)]}
=0,
\end{align*}
which implies \eqref{MM} is valid for suitably small $\varepsilon$. Therefore, for each $\lambda\in(0,\lambda_{4,2})$,
\eqref{Wd} is valid for suitably small $\varepsilon$, which, together with \eqref{t00}, implies, for $\varepsilon$ suitably small, that
\begin{eqnarray*}\label{}
\sup_{t\geq0}\psi_{\widetilde{u}_{\varepsilon,\lambda}}(t)<c(S_\lambda).
\end{eqnarray*}
Fixing such an $\varepsilon>0$ and letting $\widetilde{u}_\lambda^*\equiv \widetilde{u}_{\varepsilon,\lambda}$,
we obtain \eqref{F1}. The proof is complete.
\end{proof}

\begin{proof}[Proof of Theorem \ref{th} $(iii)$]
With the help of Lemmas \ref{lem2.1} and \ref{r}, we can prove Theorem \ref{th} $(iii)$ in
a way similar to that of Theorem \ref{th} $(i)$ and $(ii)$, and we only sketch the outline.
First, by recalling \eqref{7.1} and \eqref{7.2} with $q=2$ and $s=4$, we see that $I_\lambda$
satisfies the mountain pass geometry around $0$ and  there exists a sequence $\{u_n\}\subset H_0^2(\Omega)$ such that $I_\lambda(u_{n})\rightarrow \widetilde{c}_0$ and $I_\lambda'(u_{n})\rightarrow 0$ in $H^{-2}(\Omega)$ as $n\rightarrow\infty$,
where
\begin{equation*}
\widetilde{c}_0=\inf_{\widetilde{\gamma}\in\widetilde{\Gamma}}
\max_{t\in[0,1]}I_\lambda(\widetilde{\gamma}(t))\
and\ \
\widetilde{\Gamma}=\{\widetilde{\gamma}\in C([0,1],H_{0}^{2}(\Omega)): \widetilde{\gamma}(0)=0, \widetilde{\gamma}(1)=t_{\widetilde{u}_\lambda^*}\widetilde{u}_\lambda^*\},
\end{equation*}
and $\widetilde{u}_\lambda^*$ is given in Lemma \ref{r} which satisfies
 $I_\lambda(t_{\widetilde{u}_\lambda^*}\widetilde{u}_\lambda^*)<0$.
In view of  Lemma \ref{r}, we have
\begin{equation*}
 \widetilde{c}_0
\leq\max\limits_{t\in[0,1]}I_\lambda(tt_{\widetilde{u}_\lambda^*}\widetilde{u}_\lambda^*)
\leq\sup\limits_{t\geq0}I_\lambda(t\widetilde{u}_\lambda^*)<c(S_\lambda).
\end{equation*}
Consequently, by Lemma \ref{lem2.1}, one sees that $I_\lambda(u)$  satisfies the $(PS)_{\widetilde{c}_0}$ condition.
Then there exists a $u$ such that $I_\lambda(u)=\widetilde{c}_0$ and $I_\lambda'(u)=0$,
i.e., $u$ is a nonnegative weak solution to problem \eqref{P1.1}.
The proof is complete.
\end{proof}


\begin{thebibliography}{xx}

\bibitem{C. O. Alves}
C. O. Alves, G. M. Figueiredo, On multiplicity and concentration of positive
solutions for a class of quasilinear problems with critical exponential growth in
$\mathbb{R}^N$, J. Differ. Equ., {\bf246}(2009), \ 1288-1311.



\bibitem{A. Ambrosetti}
A. Ambrosetti, H. Br\'{e}zis, G. Cerami, Combined effects of concave and convex nonlinearities in some elliptic problems,
J. Funct. Anal.,  {\bf122}(1994), \ 519-543.


\bibitem{H. Br}
H. Br\'{e}zis, E. Lieb, A relation between pointwise convergence of functions and convergence of functionals,
Proc. Amer. Math. Soc., {\bf88}(1983), \  486-490.


\bibitem{D. Cao}
D. Cao, S. Peng, A note on the sign-changing solutions to elliptic problems
with critical Sobolev and Hardy terms, J. Differ. Equ., {\bf193}(2003), \ 424-434.


\bibitem{L. D'Ambrosio}
L. D'Ambrosio, E. Jannelli, Nonlinear critical problems for the biharmonic operator
with Hardy potential, Calc. Var. Partial Differential Equations, {\bf54}(2015), \ 365-396.


\bibitem{A. Fiscella}
A. Fiscella, E. Valdinoci, A critical Kirchhoff type problem involving a nonlocal operator,
Nonlinear Anal., {\bf94}(2014),  \ 156-170.


\bibitem{X. He}
X. He, W. Zou, Existence and concentration behavior of positive solutions for a Kirchhoff equation in $\mathbb{R}^3$,
J. Differ. Equ., {\bf252}(2012), \ 1813-1834.


\bibitem{Kang}
D. Kang, Y. Deng, Sobolev-Hardy inequality and some biharmonic problems(in Chinese),
Acta Math. Sci.,  {\bf23}(2003), \ 106-114.


\bibitem{D2}
D. Kang, L. Xu, Asymptotic behavior and existence results for the biharmonic problems involving Rellich potentials,
J. Math. Anal. Appl., {\bf455}(2017), \ 1365-1382.


\bibitem{G. Li}
G. Li, T. Yang, The existence of a nontrivial weak solution to a double critical problem involving a fractional
Laplacian in $\mathbb{R}^N$ with a Hardy term, Acta Math. Sci., {\bf40}(2020), \ 1808-1830.


\bibitem{li}
G. Li, T. Yang, L. Huang, Existence of two weak solutions for biharmonic equations with the Hardy-Sobolev critical
exponent and non-homogeneous perturbation term(in Chinese), Scientia Sinica Mathematica, {\bf49}(2019), \ 1813-1844.


\bibitem{2}
D. Naimen, The critical problem of Kirchhoff type elliptic equations in dimension four,
J. Differ. Equ., {\bf257}(2014), \ 1168-1193.


\bibitem{1}
 D. Naimen, M. Shibata, Two positive solutions for the Kirchhoff type elliptic problem
 with critical nonlinearity in high dimension, Nonlinear Anal., {\bf186}(2019), \ 187-208.


\bibitem{J. Wang}
J. Wang, L. Tian, J. Xu, F. Zhang, Multiplicity and concentration of positive solutions for a Kirchhoff
type problem with critical growth, J. Differ. Equ., {\bf253}(2012), \ 2314-2351.


\bibitem{X. Wang}
X. Wang, P. Zhao,  L. Zhang, The existence and multiplicity of classical positive solutions for a singular
nonlinear elliptic problem with any growth exponents, Nonlinear Anal., {\bf101}(2014), \ 37-46.


\bibitem{M. Willem}
M. Willem,  Minimax theorems, Birkh\"{a}user, Boston, 1996.


\bibitem{X. Mingqi}
M. Xiang, V. D. R\u{a}dulescu, B. Zhang, Fractional Kirchhoff problems with critical Trudinger-Moser nonlinearity,
Calc. Var. Partial Differential Equations, {\bf58}(2019), \ 1-27.


\bibitem{Q. Xie}
Q. Xie, X. Wu, C. Tang, Existence and multiplicity of solutions for Kirchhoff type problem with critical exponent,
Commun. Pure Appl. Anal.,
{\bf12}(2013),  \ 2773-2786.


\bibitem{B. J. Xuan}
B. Xuan, Z. Chen, Existence, multiplicity,  and bifurcation for critical polyharmonic equations,
J. Syst. Sci. Complex.,  {\bf12}(1999), \ 59-69.


\bibitem{4Hardylinj}
Y. Yao, R. Wang, Y. Shen, Nontrivial solution for a class of semilinear biharmonic equation
involving critical exponents, Acta Math. Sci., {\bf27}(2007), \ 509-514.


\bibitem{X. Zhong}
X. Zhong, C. Tang, Multiple positive solutions to a Kirchhoff type problem involving a critical nonlinearity,
Comput. Math. Appl., {\bf72}(2016), \ 2865-2877.

\end{thebibliography}
\end{document}